\documentclass[reqno]{amsart}
\usepackage{mathrsfs}
\usepackage{color}
\usepackage{amsmath}
\usepackage{amsfonts}
\usepackage{amssymb}
\usepackage{graphicx}%
\usepackage{hyperref}

 \newtheorem{Theorem}{Theorem}[section]
 \newtheorem{Corollary}[Theorem]{Corollary}
 
 \newtheorem{Proposition}[Theorem]{Proposition}

 \newtheorem{Remark}[Theorem]{Remark}

 \numberwithin{equation}{section}

\begin{document}

\title[multiplier ideal sheaves, complex singularity exponents]
 {multiplier ideal sheaves, complex singularity exponents, and restriction formula}

\author{Qi'an Guan}
\address{Qi'an Guan:  School of Mathematical Sciences, and Beijing International Center for Mathematical Research,
Peking University, Beijing, 100871, China.}
\email{guanqian@math.pku.edu.cn}
%\author{Xiangyu Zhou}
%\address{Xiangyu Zhou: Institute of Mathematics, AMSS, and Hua Loo-Keng Key Laboratory of Mathematics, Chinese Academy of Sciences, Beijing, China}
%\email{xyzhou@math.ac.cn}

\thanks{The author was partially supported by NSFC-11431013 and NSFC-11522101}

\subjclass[2010]{}
%32D15, 32E10, 32L10, 32U05, 32W05


\date{\today}

\dedicatory{}

\commby{}

%%% ----------------------------------------------------------------------

\begin{abstract}
In this article,
we obtain two sharp equality conditions in the restriction formula on complex singularity exponents:
an equality between the codimension of the zero variety of related multiplier ideal sheaves and
the relative codimension of the restriction of the variety on the submanifold (in the restriction formula);
an equivalence between the transversality (between the variety
and the submanifold) and the regularity of the restriction of the variety.
As applications,
we present sharp equality conditions in
the fundamental subadditivity property on complex singularity exponents.
\end{abstract}

%%% ----------------------------------------------------------------------
\maketitle
%%% ----------------------------------------------------------------------

\section{background, main results and applications}\label{sec:background}

Complex singularity exponent (or log canonical threshold (lct) in algebraic geometry)
is an important concept related to
the multiplier ideal sheaves in complex geometry and complex algebraic geometry.
Several important and fundamental related results have been established, e.g.
the restriction formula and the subadditivity property on complex singularity exponents
(see \cite{D-K01} \cite{demailly2010}).

Recently,
by giving lower bounds of the dimensions of the zero varieties of related multiplier ideal sheaves,
sharp equality conditions in the restriction formula and the subadditivity property
have been established \cite{GZjump-equ}.
In the present article,
continuing the above work,
we obtain two sharp equality conditions:
an equality between the codimension of the zero variety and
the codimension of the restriction of the variety on the submanifold (in the restriction formula);
an equivalence between
the transversality
(between the zero variety and the submanifold)
and the regularity of the restriction of the zero variety.
As applications, we present sharp equality conditions in
the fundamental subadditivity property.

\textbf{Organization:}
We organize the present article as follows:
in Section \ref{sec:background},
we recall some related results of multiplier ideal sheaves and complex singularity exponents,
and present the main results of the present article
(Theorem \ref{thm:jump_equality_dim_sing}, Theorem \ref{thm:lct_slice_graph})
and applications
(Proposition \ref{prop:add_dim_nonregular} and
Proposition \ref{prop:lct_add_graph});
in Section \ref{sec:prepara},
we recall some know results and present some preparatory results;
in Section \ref{sec:proof_main},
we prove the main results and applications.

\subsection{Multiplier ideal sheaves and complex singularity exponents}

Let $\Omega$ be a domain in $\mathbb{C}^{n}$ and $o\in\Omega$.
Let $u$ be a plurisubharmonic function on $\Omega$.
Nadel \cite{Nadel90} introduced
the multiplier ideal sheaf $\mathcal{I}(u)$
which can be defined as the sheaf of germs of holomorphic functions $f$ such that
$|f|^{2}e^{-2u}$ is locally integrable (background see e.g. \cite{siu96,demailly-note2000,siu00,siu05,Lazar04I,Lazar04II,siu09,demailly2010}).
Here $u$ is regarded as the weight of $\mathcal{I}(u)$.
As in \cite{demailly2010}, we denote the zero variety $\{z|\mathcal{I}(u)_{z}\neq \mathcal{O}_{z},z\in\Omega\}$
of $\mathcal{I}(u)$ by $V(\mathcal{I}(u))$.

It is well-known that the multiplier ideal sheaf $\mathcal{I}(u)$ is coherent and integral closed, satisfies Nadel's vanishing theorem \cite{Nadel90},
and the strong openness property $\mathcal{I}(u)=\cup_{\varepsilon>0}\mathcal{I}((1+\varepsilon)u)$ \cite{GZopen-a,GZopen-b,GZopen-c}
i.e. the solution of Demailly's strong openness conjecture
(the background and motivation of the conjecture could be referred to \cite{demailly-note2000,demailly2010},
the proof of $\dim\leq 2$ case could be referred to \cite{JM12,JM13}).

The complex singularity exponent (see \cite{tian87}, see also \cite{demailly-note2000,demailly2010},
lct in algebraic geometry see \cite{Sho92,Ko92}) is defined
$$c_{o}(u):=sup\{c\geq0:\exp{(-2cu)}\,\,\text{is integrable near}\,\,o\}.$$
Berndtsson's solution (\cite{berndtsson13}) of the openness conjecture posed by Demailly-Kollar, i.e. $e^{-2c_{o}(u)u}$ is not integrable near $o$,
implies that $\{z|c_{z}(u)\leq c\}=V(\mathcal{I}(cu))$ is an analytic set.

Let $I\subseteq \mathcal{O}_{o}$ be a coherent ideal.
The jumping number $c_{o}^{I}(u)$ is defined
(see e.g. \cite{JM12,JM13})
$c_{o}^{I}(u):=sup\{c\geq0:|I|^{2}\exp{(-2cu)}\,\,\text{is integrable near}\,\,o\}$.
Let $\mathcal{F}\subseteq\mathcal{O}$ be a coherent ideal sheaf.
In \cite{GZjump-equ},
it has been presented that the strong openness property $\mathcal{I}(u)=\cup_{\varepsilon>0}\mathcal{I}((1+\varepsilon)u)$
implies that $\{z|c_{z}^{\mathcal{F}_{z}}(u)\leq p\}=Supp(\mathcal{F}/(\mathcal{F}\cap\mathcal{I}(pu)))$
is an analytic subset.

\subsection{Main results: Sharp equality conditions in the restriction formula on complex singularity exponents}
Let $u$ be a plurisubharmonic function near $o\in\mathbb{C}^{n}$.
Let $p:\mathbb{C}^{n}\to\mathbb{C}^{n-k}$
satisfying $p(z_{1},\cdots,z_{n})=(z_{k+1},\cdots,z_{n}),$
where $(z_{1},\cdots,z_{n})$ and $(z_{k+1},\cdots,z_{n})$ are coordinates of
$\mathbb{C}^{n}$ and $\mathbb{C}^{n-k}$ respectively.
Let $H=\{z_{k+1}=\cdots=z_{n}=0\}$.
In \cite{D-K01}, the following restriction formula (an "important monotonicity result" as in \cite{D-K01}) on complex singularity exponents
has been presented by using Ohsawa-Takegoshi $L^{2}$ extension theorem.

\begin{Proposition}
\label{prop:DK2000}\cite{D-K01}
$c_{o}(u|_{H})\leq c_{o}(u)$
holds,
where $u|_{H}\not\equiv-\infty$.
\end{Proposition}

Let $A=V(\mathcal{I}(c_{o}(u)u))$.
In \cite{GZjump-equ}, the following equality condition in Proposition \ref{prop:DK2000} has been established

\begin{Theorem}
\label{thm:jump_equality}\cite{GZjump-equ}
If $c_{o}(u|_{H})=c_{o}(u)$,
then $\dim_{o}A\geq n-k$.
\end{Theorem}

For the case $k=1$ and $n=2$, Theorem \ref{thm:jump_equality} can be referred to \cite{B-M} and \cite{FM05j}.

For the case $k=1$ and $n>2$, Theorem \ref{thm:jump_equality} can be referred to \cite{GZopen-lelong}
(a recent new proof by combining methods in \cite{GZopen-lelong} and \cite{demailly-Pham} can be referred to \cite{Rash1501}).

In the present article,
combining recent results in \cite{GZjump-equ} with some new ideas,
we obtain the following sharp equality condition in Proposition \ref{prop:DK2000}
by giving the equality between $n-\dim_{o}A$ and $k-\dim_{o}(A\cap H)$.
\begin{Theorem}
\label{thm:jump_equality_dim_sing}
If $c_{o}(u|_{H})=c_{o}(u)$,
then $\dim_{o}A=n-k+\dim_{o}(A\cap H)$.
\end{Theorem}

It is known that for the case $k=1$ and any $n$,
one can obtain the regularity of $(A,o)$ (see \cite{GZopen-lelong} and \cite{GZjump-equ}).
Then it is natural to consider the regularity of $(A,o)$ for general $k$.
When $n=2$, and $u=\log|z_{2}-z_{1}^{2}|$ $(H=\{z_{2}=0\})$, one can obtain that $c_{o}(u|_{H})=1/2<1=c_{o}(u)$.
It is natural to consider the transversality between $(A,o)$ and $(H,o)$.

In the present article,
we obtain the following sharp equality condition in Proposition \ref{prop:DK2000}
by giving the equivalence between the transversality (between $(A,o)$ and $(H,o)$)
and the regularity of $(A\cap H,o)$.

\begin{Theorem}
\label{thm:lct_slice_graph}
If $c_{o}(u|_{H})=c_{o}(u)$,
then the following statements are equivalent

$(1)$ $(A\cap H,o)$ is regular;

$(2)$ there exist coordinates $(w_{1},\cdots,w_{k},z_{k+1},\cdots,z_{n})$ near $o$ and $l\in\{1,\cdots,k\}$
such that $(A,o)=(w_{1}=\cdots=w_{l}=0,o)$;

$(3)$ there exist coordinates $(w_{1},\cdots,w_{k},z_{k+1},\cdots,z_{n})$ near $o$ and $l\in\{1,\cdots,k\}$
such that $\mathcal{I}(c_{o}(u)u)_{o}=(w_{1},\cdots,w_{l})_{o}$.
\end{Theorem}

\subsection{Applications:
Sharp equality conditions in the fundamental subadditivity property of complex singularity exponents}

Let $u$ and $v$ be plurisubharmonic functions near $o\in\mathbb{C}^{n}$.
In \cite{D-K01,GZjump-equ},
the fundamental subadditivity property of complex singularity exponents
has been presented.

\begin{Theorem}
\label{thm:subadd_cse_general}\cite{D-K01,GZjump-equ} $c_{o}(\max\{u,v\})\leq c_{o}(u)+c_{o}(v).$
\end{Theorem}

Let $c=c_{o}(u)+c_{o}(v)$.
Let $A_{1}=V(\mathcal{I}(cu))$ and $A_{2}=V(\mathcal{I}(cv))$.
In \cite{GZjump-equ}, the following sharp equality condition in Theorem \ref{thm:subadd_cse_general}
has been established.
\begin{Theorem}
\label{thm:equality_subadd_lct}\cite{GZjump-equ}
If $c_{o}(\max\{u,v\})=c$,
then $\dim_{o}A_{1}+\dim_{o}A_{2}\geq n$.
\end{Theorem}

Let $B=\{z|c_{z}(u)+c_{z}(v)\leq c\}$,
which is an analytic subset on $A_{1}\cap A_{2}$ (see subsection \ref{sec:add_nonregular}).

As an application of Theorem \ref{thm:jump_equality_dim_sing},
we present the following more precise version of Theorem \ref{thm:equality_subadd_lct}

\begin{Proposition}
\label{prop:add_dim_nonregular}
If $c_{o}(\max\{u,v\})=c$,
then $\dim_{o}A_{1}+\dim_{o}A_{2}\geq n+\dim_{o}B$.
\end{Proposition}

Let $n=2$, $u=\log|z_{1}|$, $v=\log|z_{1}-z_{2}^{2}|$.
As
$(|z_{1}|+|z_{2}^{2}|)/6\leq\max\{|z_{1}|,|z_{1}-z_{2}^{2}|\}\leq6(|z_{1}|+|z_{2}^{2}|)$,
it is clear that $c_{o}(\max\{u,v\})=1+1/2<2=c$,
$A_{1}\cap A_{2}=\{o\}$.
Then it is natural to consider the transversality between $A_{1}$ and $A_{2}$.

As an application of Theorem \ref{thm:lct_slice_graph},
we present
the following sharp equality condition in Theorem \ref{thm:subadd_cse_general}
by giving the regularity of $A_{1}$ and $A_{2}$ and the transversality between $A_{1}$ and $A_{2}$.

\begin{Proposition}
\label{prop:lct_add_graph}
Assume that $(A_{1},o)$ and $(A_{2},o)$ are both irreducible
such that $(B,o)=(A_{1}\cap A_{2},o)$, which is regular.
If $c_{o}(\max\{u,v\})=c$,
then both $(A_{1},o)$ and $(A_{2},o)$ are regular such that $\dim (T_{A_{1},o}+T_{A_{2},o})=n$.
\end{Proposition}

\section{Some preparations}\label{sec:prepara}

In this section, we recall some known results and present some preparatory results
for the proofs of the main results and applications in the present article.

\subsection{Ohsawa-Takegoshi $L^{2}$ extension theorem}

We recall the famous Ohsawa-Takegoshi $L^2$ extension theorem as follows:

\begin{Theorem}
\label{t:ot_plane}(\cite{ohsawa-takegoshi}, see also \cite{ohsawa3,siu96,berndtsson,demailly99,siu00}, etc.)
Let $D$ be a bounded pseudo-convex domain in $\mathbb{C}^{n}$.
Let $u$ be a plurisubharmonic function on $D$.
Let $H$ be an $m$-dimensional complex plane in $\mathbb{C}^{n}$.
Then for any holomorphic function $f$ on $H\cap D$ satisfying
$$\int_{H\cap D}|f|^{2}e^{-2u}d\lambda_{H}<+\infty,$$
there exists a holomorphic function $F$ on $D$ satisfying $F|_{H\cap D}=f$,
and
$$\int_{D}|F|^{2}e^{-2u}d\lambda_{n}\leq C_{D}\int_{H\cap D}|f|^{2}e^{-2u}d\lambda_{H},$$
where $C_{D}$ only depends on the diameter of $D$ and $m$,
and $d\lambda_{H}$ is the Lebesgue measure on $H$.
\end{Theorem}

The optimal estimates of generalized versions of Theorem \ref{t:ot_plane}
could be referred to \cite{guan-zhou12,guan-zhou13p,guan-zhou13ap}.

Following the symbols in Theorem \ref{t:ot_plane}, there is a local version of Theorem \ref{t:ot_plane}

\begin{Remark}
\label{rem:ot_plane}(see \cite{ohsawa-takegoshi}, see also \cite{D-K01})
For any germ of holomorphic function $f$ on $o\in H\cap D$ satisfying
$|f|^{2}e^{-2u|_{H}}$ is locally integrable near $o$,
there exists a germ of holomorphic function $F$ on $o\in D$ satisfying $F|_{H\cap D}=f$,
and $|F|^{2}e^{-2u}$ is locally integrable near $o$.
\end{Remark}

\subsection{Berndtsson's solution of the openness conjecture}

Let $(z_{1},\cdots,z_{k})$ be the coordinates of
$\mathbb{B}^{k-l}\times\mathbb{B}^{l}\subseteq\mathbb{C}^{k}$,
and let
$p:\mathbb{B}^{k-l}\times\mathbb{B}^{l}\to\mathbb{B}^{k-l}$.
Let $H_{1}:=\{z_{k-l+1}=\cdots=z_{k}=0\}$.
In \cite{GZjump-equ},
by combining with Theorem \ref{t:ot_plane},
Berndtsson's solution of the openness conjecture \cite{berndtsson13}
and Berndtsson's log-plurisubharmonicity of
relative Bergman kernels \cite{bern_bergman},
it has been presented
\begin{Remark}
\label{rem:lct_slice_open}\cite{GZjump-equ}
Let $u$ be a plurisubharmonic function on $\mathbb{B}^{k-l}\times\mathbb{B}^{l}\subseteq\mathbb{C}^{k}$.
Then there exists $c\in(0,+\infty]$,
such that $c_{z_{a}}(u|_{L_{a}})=c$ for almost every
$a=(a_{1},\cdots,a_{k-l})\in\mathbb{B}^{k-l}$
with respect to the Lebesgue measure on $\mathbb{C}^{k-l}$,
where $L_{a}=\{z_{1}=a_{1},\cdots,z_{k-l}=a_{k-l}\}$
and $z_{a}\in L_{a}\cap H_{1}$.
\end{Remark}

We present a corollary of Remark \ref{rem:lct_slice_open} as follows

\begin{Corollary}
\label{coro:lct_slice_open}
Let $u$ be a plurisubharmonic function on $\mathbb{B}^{k-l}\times\mathbb{B}^{l}$.
Assume that $c_{z}(u)\leq 1$ for any $z\in H_{1}$ and $c_{o}(u)=1$.
Then for almost every
$a=(a_{1},\cdots,a_{k-l})\in\mathbb{B}^{k-l}$
with respect to the Lebesgue measure on $\mathbb{C}^{k-l}$, $c_{z_{a}}(u|_{L_{a}})=1$ holds,
where $L_{a}$
and $z_{a}$ are as in Remark \ref{rem:lct_slice_open}.
\end{Corollary}

\begin{proof}
By Remark \ref{rem:lct_slice_open} and $c_{o}(u)=1$,
it follows that $c\geq 1$ (consider the integrability of $e^{-2pu}$ near $o$, where $p<1$ near $1$, and by contradiction).

By $c_{z}(u)\leq 1$ for any $z\in H_{1}$ and Proposition \ref{prop:DK2000},
it follows that $c_{z_{a}}(u|_{L_{a}})\leq c_{z_{a}}(u)\leq1$ for any $z_{a}\in L_{a}\cap H_{1}$.
Combining $c\geq 1$, we obtain Corollary \ref{coro:lct_slice_open}.
\end{proof}

The following remark is the singular version of Corollary \ref{coro:lct_slice_open}:

\begin{Remark}
\label{rem:lct_slice_sing}
Let $A_{3}$ be a reduced irreducible analytic subvariety on $\mathbb{B}^{k-l}\times\mathbb{B}^{l}$ through $o$
satisfying $\dim_{o}A_{3}=k-l$
such that

$(1)$ for any $a=(a_{1},\cdots,a_{k-l})\in\mathbb{B}^{k-l}$,
$A_{3}\cap L_{a}\neq\emptyset$, where $L_{a}=\{z_{1}=a_{1},\cdots,z_{k-l}=a_{k-l}\}$;

$(2)$ there exists analytic subset $A_{4}\subseteq\mathbb{B}^{k-l}$
such that any $z\in (A_{3}\cap p^{-1}(\mathbb{B}^{k-l}\setminus A_{4}))$
is the regular point in $A_{3}$ and the noncritical point of $p|_{A_{3,reg}}$.

Let $u$ be a plurisubharmonic function on $\mathbb{B}^{k-l}\times\mathbb{B}^{l}$.
Assume that $c_{z}(u)\leq 1$ for any $z\in A_{3}$ and $c_{o}(u)=1$.
Then for almost every
$a=(a_{1},\cdots,a_{k-l})\in\mathbb{B}^{k-l}$
with respect to the Lebesgue measure on $\mathbb{C}^{k-l}$,
there exists $z_{a}\in A_{3}\cap L_{a}$
such that
equality $c_{z_{a}}(u|_{L_{a}})=1$.
\end{Remark}

\begin{proof}
By Remark \ref{rem:lct_slice_open},
it follows that there exists $c\in(0,\infty]$.
such that $c_{z}(u|_{L_{p(z)}})=c$
for almost every $z\in (A_{3}\cap p^{-1}(\mathbb{B}^{k-l}\setminus A_{4}))$
with respect to the Lebesgue measure
in $(A_{3}\cap p^{-1}(\mathbb{B}^{k-l}\setminus A_{4}))$.
By $c_{o}(u)=1$,
it follows that $c\geq 1$
(consider the integrability of $e^{-2pu}$ near $o$, where $p<1$ near $1$,
and by contradiction).

By Proposition \ref{prop:DK2000},
it follows that $1\leq c\leq c_{z}(u|_{L_{p(z)}})\leq c_{z}(u)\leq 1$ holds,
for almost every $z\in (A_{3}\cap p^{-1}(\mathbb{B}^{k-l}\setminus A_{4}))$.
Then one can find $z_{3}\in (A_{3}\cap p^{-1}(\mathbb{B}^{k-l}\setminus A_{4}))$
such that $c_{z_{3}}(u)=c_{z_{3}}(u|_{L_{p(z_{3})}})=1$.

By Corollary \ref{coro:lct_slice_open} $(o\sim z_{3})$,
it follows that
for almost every
$a=(a_{1},\cdots,a_{k-l})\in(\mathbb{B}^{k-l}\setminus A_{4})$
with respect to the Lebesgue measure on $\mathbb{C}^{k-l}$,
there exists $z_{a}\in A_{3}\cap L_{a}$
such that
equality $c_{z_{a}}(u|_{L_{a}})=1$.
As the Lebesgues measure of $A_{4}$ on $\mathbb{C}^{k-l}$ is zero,
then we obtain Remark \ref{rem:lct_slice_sing}.
\end{proof}

\subsection{Hilbert's Nullstellensatz (complex situation)}

It is well-known that the complex situation of Hilbert's Nullstellensatz is as follows (see (4.22) in \cite{demailly-book})

\begin{Theorem}
\label{thm:hilbert}(see \cite{demailly-book})
For every ideal $I\subset \mathcal{O}_{o}$,
$\mathcal{J}_{V(I),o}=\sqrt{I},$
where $\sqrt{I}$ is the radical of $I$, i.e. the set of germs $f\in\mathcal{O}_{o}$ such that some power $f^{k}$ lies in $I$.
\end{Theorem}

\subsection{A consequence of Demailly's strong openness conjecture}

In \cite{GZopen-lelong}, inspired by the proof of Demailly's equisingular approximation theorem
(see Theorem 15.3 in \cite{demailly2010}) and using the
strong openness property $\mathcal{I}(u)=\cup_{\varepsilon>0}\mathcal{I}((1+\varepsilon)u)$ \cite{GZopen-a,GZopen-b,GZopen-c},
the following observation has been presented:

\begin{Proposition}
\label{Pro:GZ1005}(\cite{GZopen-lelong}, see also \cite{GZjump-equ})
Let $D$ be a bounded domain in $\mathbb{C}^{n}$, and $o\in D$.
Let $u$ be a plurisubharmonic function on $D$.
Let $J\subseteq\mathcal{I}(u)_{o}$.
Then there exists $p>0$ large enough, such that the plurisubharmonic function
$\tilde{u}=\max\{u,p\log|J|\}$ on a small enough neighborhood $V_{o}$ of $o$ satisfying that
$e^{-2u}-e^{-2\tilde{u}}$ is integrable.
\end{Proposition}

We recall some direct consequences of Proposition \ref{Pro:GZ1005} as follows
\begin{Remark}
\label{rem:GZ1005}(\cite{GZopen-lelong}, see also \cite{GZjump-equ})
$\tilde{u}$ (as in Proposition \ref{Pro:GZ1005}) satisfies:

$(1)$ for any $z\in(\{z|c_{z}(u)\leq1\},o)=(V(\mathcal{I}(u)),o)$, inequality $c_{z}(u)\leq c_{z}(\tilde{u})\leq 1$ holds;

$(2)$ if $c_{z_{0}}(u)=1$, then $c_{z_{0}}(\tilde{u})=1$, where $z_{0}\in(\{z|c_{z}(u)\leq1\},o)$.
\end{Remark}

Let $I\subseteq\mathcal{O}_{o}$ be a coherent ideal,
and let $u$ be a plurisubharmonic function near $o$.
Using Proposition \ref{Pro:GZ1005},
we present the following result about the integrability of the ideals related to weight of jumping number one.

\begin{Proposition}
\label{thm:jump_supp}
Let $J\subseteq\mathcal{O}_{o}$ be a coherent ideal.
Assume that $c_{o}^{I}(u)=1$.
If $(V(\mathcal{I}(u)),o)\subseteq(V(J),o)$,
then $|I|^{2}|J|^{2\varepsilon}e^{-2u}$ is locally integrable near $o$ for any $\varepsilon>0$.
\end{Proposition}
After the present article has been written,
Demailly kindly shared his manuscript \cite{demailly2015} with the author addresses,
which includes Proposition \ref{thm:jump_supp} (Lemma (4.2) in \cite{demailly2015}).

\begin{proof}(proof of Proposition \ref{thm:jump_supp})
Let $J_{0}\subseteq\mathcal{O}_{o}$ be a coherent ideal satisfying
$(V(J_{0}),o)\supseteq(V(\mathcal{I}(u)),o).$
By Theorem \ref{thm:hilbert} $(I\sim \mathcal{I}(u)_{o})$,
it follows that there exists large enough positive integer $N$
such that $J^{N}_{0}\subseteq\mathcal{I}(u)_{o}$.

By Proposition \ref{Pro:GZ1005},
it follows that exist $p_{0}>0$ large enough
such that
$e^{-2u}-e^{-2\max\{u,p_{0}\log|J_{0}|\}}$
is locally integrable near $o$.

It suffices to prove that
$|I|^{2}|J_{0}|^{2\varepsilon}e^{-2\max\{u,p_{0}\log|J_{0}|\}}$ is locally integrable near $o$ for small enough $\varepsilon>0$.
We prove the above statement by contradiction:
If not, then there exists $\varepsilon_{0}>0$,
such that $|I|^{2}|J_{0}|^{2\varepsilon_{0}}e^{-2\max\{u,p_{0}\log|J_{0}|\}}$ is not locally integrable near $o$.
Note that $\varepsilon_{0}\log|J_{0}|\leq \frac{\varepsilon_{0}}{p_{0}}\max\{u,p_{0}\log|J_{0}|\}$,
then it follows that
$|I|^{2}e^{-2(1-\frac{\varepsilon_{0}}{p_{0}})\max\{u,p_{0}\log|J_{0}|\}}$ is not locally integrable near $o$.
Note that $u\leq\max\{u,p_{0}\log|J_{0}|\}$,
then it follows that $|I|^{2}e^{-2(1-\frac{\varepsilon_{0}}{p_{0}})u}$ is not locally integrable near $o$,
which contradicts $c_{o}^{I}(u)=1$.
Then we prove Proposition \ref{thm:jump_supp}.
\end{proof}
Let $I=\mathcal{O}_{o}$, $\varepsilon=1$.
Using Proposition \ref{thm:jump_supp},
we obtain the following result.

\begin{Corollary}
\label{coro:jump_supp}Let $J\subseteq\mathcal{O}_{o}$ be a coherent ideal.
Assume that $c_{o}(u)=1$.
Then the following two statements are equivalent\\
$(1)$ $(V(J),o)\supseteq(V(\mathcal{I}(u)),o)$;\\
$(2)$ $|J|^{2}e^{-2u}$ is locally integrable near $o$, i.e. $J\subseteq \mathcal{I}(u)_{o}$.
\end{Corollary}

\subsection{Product spaces}

Let $\pi_{i}:=\Omega_{1}\times\Omega_{2}\to\Omega_{i}$
$i=1,2$  be the projections,
where $\Omega_{i}\subset\mathbb{C}^{n}$ and containing the origin $o\in\mathbb{C}^{n}$ for any $i$.
Let $\Delta$ be the diagonal of $\mathbb{C}^{n}\times\mathbb{C}^{n}$.
It is well-known that

\begin{Remark}\label{lem:linear}
Let $A_{1}$ and $A_{2}$ be two varieties on $\Omega_{1}$ and $\Omega_{2}$ respectively through $o$.
Assume that $A_{1}$ and $A_{2}$ are both regular at $o$.
Then $\dim (T_{A_{1},o}\cap T_{A_{2},o})=\dim(T_{A_{1}\times A_{2},(o,o)}\cap T_{\Delta,(o,o)})$.
\end{Remark}

Let $u$ and $v$ be plurisubharmonic functions near $o$.
In \cite{D-K01,GZjump-equ},
the following statement has been presented, which was used to prove Theorem \ref{thm:subadd_cse_general}.

\begin{Proposition}
\label{prop:add_prod_cse}\cite{D-K01,GZjump-equ}
$c_{z\times w}(\max\{\pi_{1}^{*}(u),\pi_{2}^{*}(v)\})=c_{z}(u)+c_{w}(v).$
\end{Proposition}

Let $c=c_{o}(u)+c_{o}(v)$.
Proposition \ref{prop:add_prod_cse} shows that $c_{(o,o)}(\max\{\pi_{1}^{*}(u),\pi_{2}^{*}(v)\})=c$.

\section{Proofs of main results and applications}\label{sec:proof_main}

In this section,
we present the proofs of main results and applications.

\subsection{Proof of Theorem \ref{thm:jump_equality_dim_sing}}

By Remark \ref{rem:GZ1005} $(u\sim c_{o}(u)u, J=\mathcal{I}(c_{o}(u)u)_{o})$,
it follows that
\begin{equation}
\label{equ:151013b}
c_{z}(\tilde{u})\leq 1
\end{equation}
for any $z\in (A,o)$
and $c_{o}(\tilde{u})=1$ $(\Leftarrow c_{o}(c_{o}(u)u)=1)$.
By Proposition \ref{prop:DK2000},
it follows that
\begin{equation}
\label{equ:151010a}
c_{z}(\tilde{u}|_{H})\leq 1
\end{equation}
for any $z\in (A\cap H,o)$.

Using Proposition \ref{prop:DK2000} and inequality \ref{equ:151010a},
one can obtain that $c_{o}(\tilde{u}|_{H})\leq c_{o}(\tilde{u})\leq 1$.
Combining with
$1=c_{o}(u)/c_{o}(u)=c_{o}(u|_{H})/c_{o}(u)=c_{o}(c_{o}(u)u|_{H})\leq c_{o}(\tilde{u}|_{H})$ $(\Leftarrow c_{o}(u)u\leq\tilde{u})$,
one can obtain that
\begin{equation}
\label{equ:151013a}
c_{o}(\tilde{u}|_{H})=c_{o}(\tilde{u})=1.
\end{equation}

Let $l=k-\dim_{o}(A\cap H)$.
Let $A_{3}$ be a irreducible component of
$A\cap H$ on $\mathbb{B}^{k-l}\times\mathbb{B}^{l}\subset H$ through $o$
satisfying $\dim_{o}A_{3}=k-l$.

By the parametrization of $(A_{3},o)$ in $H$ (see "Local parametrization theorem" (4.19) in \cite{demailly-book}),
it follows that
one can find local coordinates $(z_{1},\cdots,z_{n})$ of
a neighborhood $U=\mathbb{B}^{k-l}\times\mathbb{B}^{l}\times\mathbb{B}^{n-k}$ of $o$ satisfying $H=\{z_{k+1}=\cdots=z_{n}=0\}$
and $\dim(A\cap U)=\dim_{o}A$
such that

$(1)$ $A_{3}\cap((\mathbb{B}^{k-l}\times\mathbb{B}^{l})\cap H)$ is reduced and irreducible;

$(2)$ for any $a=(a_{1},\cdots,a_{k-l})\in\mathbb{B}^{k-l}$,
$A_{3}\cap L_{a}\neq\emptyset$, where $L_{a}=\{z_{1}=a_{1},\cdots,z_{k-l}=a_{k-l}\}$;

$(3)$ there exists analytic subset $A_{4}\subseteq\mathbb{B}^{k-l}$
such that any $z\in (A_{3}\cap p^{-1}(\mathbb{B}^{k-l}\setminus A_{4}))$
is the regular point in $A_{3}$ and the noncritical point of $p|_{A_{3,reg}}$,
where $p:(z_{1},\cdots,z_{k})=(z_{1},\cdots,z_{k-l})$.

By $(2)$, $(3)$, $c_{o}(\tilde{u})=c_{o}(\tilde{u}|_{H})=1$ (inequality \ref{equ:151013a}),
$c_{z}(\tilde{u})\leq1$ for any $z\in A_{3}$ (inequality \ref{equ:151013b}),
and Remark \ref{rem:lct_slice_sing},
it follows that for almost every
$a=(a_{1},\cdots,a_{k-l})\in\mathbb{B}^{k-l}$
with respect to the Lebesgue measure on $\mathbb{C}^{k-l}$,
there exists $z_{a}\in A_{3}\cap L_{a}$
such that
equality $c_{z_{a}}(\tilde{u}|_{L_{a}})=1$ (the set of $a$ denoted by $A_{ae}$).

Let $\tilde{L}_{a}=\{z_{1}=a_{1},\cdots,z_{k-l}=a_{k-l}\}$.
By inequality \ref{equ:151013b} and Proposition \ref{prop:DK2000},
it follows that for any $a\in A_{ae}$,
$1=c_{z_{a}}(\tilde{u}|_{L_{a}})\leq c_{z_{a}}(\tilde{u}|_{\tilde{L}_{a}})\leq c_{z_{a}}(\tilde{u})=1$,
which implies $c_{z_{a}}(\tilde{u}|_{L_{a}})=c_{z_{a}}(\tilde{u}|_{\tilde{L}_{a}})=1$.
Using
Theorem \ref{thm:jump_equality}
($\mathbb{C}^{n}\sim\tilde{L}_{a}$,
$H\sim H\cap\tilde{L}_{a}=L_{a}$,
$o\sim z_{a}$, $u\sim\tilde{u}|_{\tilde{L}_{a}}$),
one can obtain that
for any $a\in A_{ae}$,
$\max_{z_{a}\in p^{-1}(a)}\dim_{z_{a}}\{z|c_{z}(\tilde{u}|_{\tilde{L}_{a}})\leq 1\}\geq n-l-(k-l)=n-k$.
By the definition of $\tilde{u}$,
it follows that
$(A\cap U)\cap \tilde{L}_{a})\supseteq\{c_{o}(\tilde{u}|_{\tilde{L}_{a}})\leq 1\}$,
which implies
$\dim((A\cap U)\cap \tilde{L}_{a})\geq\max_{z_{a}\in p^{-1}(a)}\dim_{z_{a}}\{z|c_{z}(\tilde{u}|_{\tilde{L}_{a}})\leq 1\}$.
Then we obtain that
the $2(n-k)$-dimensional Hausdorff measure
of $\dim((A\cap U)\cap \tilde{L}_{a})$ is not zero for any $a\in A_{ae}$.

Note that the $2(k-l)$-dimensional Hausdorff measure of $A_{ae}$ is not zero,
then it follows that the $2(n-k)+2(k-l)=2(n-l)$-dimensional Hausdorff measure
of $A$ near $o$ is not zero (see Theorem 3.2.22 in \cite{Fed69}),
which implies that
$\dim_{o}A=\dim(A\cap U)\geq n-l$.
Note that $l=k-\dim_{o}(A\cap H)$ implies $\dim_{o}A\leq n-k+(k-l)=n-l$,
then Theorem \ref{thm:jump_equality_dim_sing} has been proved.

\subsection{Proof of Theorem \ref{thm:lct_slice_graph}}
By Corollary \ref{coro:jump_supp},
it follows that $(2)\Leftrightarrow (3)$.

In order to prove Theorem \ref{thm:lct_slice_graph},
by Theorem \ref{thm:jump_equality_dim_sing},
it suffices to prove the following statement $((1)\Rightarrow(2))$.\\

Assume that $(A\cap H,o)$ is regular,
and
$k-\dim_{o}A\cap H=n-\dim_{o}A$.
If $c_{o}(u)=c_{o}(u|_{H})$,
then there exist coordinates $(w_{1},\cdots,w_{k},z_{k+1},\cdots,z_{n})$ near $o$ and $l\in\{1,\cdots,k\}$,
such that $(w_{1}=\cdots=w_{l}=0,o)=(A,o).$\\

Let $J_{0}=\mathcal{I}(c_{o}(u)u)_{o}$.
By Remark \ref{rem:GZ1005} $(u\sim c_{o}(u)u)$,
it follows that there exists $p_{0}>0$ large enough,
such that $\tilde{u}:=\max\{c_{o}(u)u,p_{0}\log|J_{0}|\}$
satisfies:
$(1)$ $c_{o}(\tilde{u})=1$ $(\Leftarrow c_{o}(c_{o}(u)u)=1)$;
$(2)$ $(\{z|c_{z}(\tilde{u})\leq 1\},o)=(A,o)$.

By $\tilde{u}|_{H}\geq c_{o}(u)u|_{H}= c_{o}(u|_{H})u|_{H}$,
it follows that
$c_{o}(\tilde{u}|_{H})\geq c_{o}(c_{o}(u)u|_{H})=c_{o}(c_{o}(u|_{H})u|_{H})=1$.
Combining with the fact that
$c_{o}(\tilde{u}|_{H})\leq c_{o}(\tilde{u})=1$,
we obtain that
\begin{equation}
\label{equ:151009a}
c_{o}(\tilde{u}|_{H})=1.
\end{equation}

Note that $c_{z}(\tilde{u}|_{H})\leq c_{z}(\tilde{u})$ for any $z\in A\cap H$,
then by $(2)$ $(\Rightarrow c_{z}(\tilde{u})\leq 1)$ for any $z\in A\cap H$,
it follows that
$(\{z|c_{z}(\tilde{u}|_{H})\leq 1\},o)\supseteq(A\cap H,o).$
Combining with the definition of
$\tilde{u}$ $(\Rightarrow (\{z|c_{z}(\tilde{u}|_{H})<+\infty\},o)\subseteq(V(J_{o})\cap H,o)=(A\cap H,o))$,
we obtain
\begin{equation}
\label{equ:151009b}
(V(\mathcal{I}(\tilde{u}|_{H})),o)=(\{z|c_{z}(\tilde{u}|_{H})\leq 1\},o)=(A\cap H,o).
\end{equation}

In the following part of the present section,
we consider $\tilde{u}$ instead of $u$.

By equality \ref{equ:151009b}, it follows that $(V(\mathcal{I}(\tilde{u}|_{H})),o)(=(A\cap H,o))$ is regular.
Combining with equality \ref{equ:151009a} and Corollary \ref{coro:jump_supp} $(u\sim \tilde{u}|_{H}$ on $(H,o))$,
it follows that there exist $l\in\{1,\cdots,k\}$ and holomorphic functions $f_{1},\cdots,f_{l}$ near $o\in H$ such that \\
$(a)$ $df_{1}|_{o},\cdots,df_{l}|_{o}$ are linear independent;\\
$(b)$ $(\{f_{1}=\cdots=f_{l}=0\},o)=(A\cap H,o)$ holds;\\
$(c)$ $|f_{j}|^{2}e^{-2\tilde{u}|_{H}}$ are all locally integrable near $o$ for $j\in\{1,\cdots,l\}$.

By Remark \ref{rem:ot_plane} and $(c)$,
it follows that there exist holomorphic functions $F_{1},\cdots,F_{l}$ near $o\in \mathbb{C}^{n}$ such that
 and $|F_{j}|^{2}e^{-2\tilde{u}}$ are integrable near $o$ for any $j\in\{1,\cdots,l\}$,
which implies that $\{F_{1}=\cdots=F_{l}=0\}\supseteq A$.
Combining $F_{j}=f_{j}$ and $(a)$, we obtain that $dF_{1}|_{o},\cdots,dF_{l}|_{o},dz_{k+1}|_{o},\cdots,dz_{n}|_{o}$ are linear independent.

Note that $\{F_{1}=\cdots=F_{l}=0\}$ is regular near $o$ and $n-\dim_{o}A=k-\dim_{o}A\cap H=l$,
then it follows that $\{F_{1}=\cdots=F_{l}=0\}=A$ near $o$.
Choosing $w_{j}=F_{j}$ for any $j\in\{1,\cdots,l\}$,
one can find holomorphic functions $w_{l+1},\cdots,w_{k}$ near $o$ such that
$dw_{1}|_{o},\cdots,dw_{k}|_{o},dz_{k+1}|_{o},\cdots,dz_{n}|_{o}$ are linear independent.
Then Theorem \ref{thm:lct_slice_graph} has been proved.

\subsection{Proof of Proposition \ref{prop:add_dim_nonregular}}\label{sec:add_nonregular}

Let $A_{1}=V(\mathcal{I}(cu))$ and $A_{2}=V(\mathcal{I}(cv))$,
and $A=\{(z,w)|c_{(z,w)}(\max\{\pi_{1}^{*}(u),\pi_{2}^{*}(v)\})\leq c\}$.
By Proposition \ref{thm:subadd_cse_general},
it follows that
\begin{equation*}
\begin{split}
c_{(o,o)}(\max\{\pi_{1}^{*}(u),\pi_{2}^{*}(v)\})
&=c_{o}(u)+c_{o}(v)=c
\\&=c_{o}(\max\{u,v\})
=c_{(z,w)}(\max\{\pi_{1}^{*}(u),\pi_{2}^{*}(v)\}|_{\Delta}).
\end{split}
\end{equation*}
Using Theorem \ref{thm:jump_equality_dim_sing} $(u\sim\max\{\pi_{1}^{*}(u),\pi_{2}^{*}(v)\}$,
$H\sim\Delta$, $o\sim (o,o)$, $k\sim n$, $n\sim 2n)$,
we obtain $\dim_{(o,o)}A=\dim_{(o,o)}(A\cap\Delta)+n$.
By Proposition \ref{prop:add_prod_cse},
it follows that
$A=\{(z,w)|c_{z}(u)+c_{w}(v)\leq c\}\subseteq\{(z,w)|\max\{c_{z}(u),c_{w}(v)\}\leq c\}=A_{1}\times A_{2}$,
which implies  $\dim_{o}A_{1}+\dim_{o}A_{2}=\dim_{(o,o)}(A_{1}\times A_{2})\geq \dim_{(o,o)}A$.
Note that $B=\{z|c_{z}(u)+c_{z}(v)\leq c\}$ is biholomophic to $A\cap\Delta$,
then it follows that $\dim_{o}A_{1}+\dim_{o}A_{2}\geq \dim_{(o,o)}A=\dim_{(o,o)}(A\cap\Delta)+n=n+\dim_{o}B$.
Proposition \ref{prop:add_dim_nonregular} has thus been proved.

\subsection{Proof of Proposition \ref{prop:lct_add_graph}}
Following the symbols in subsection \ref{sec:add_nonregular},
by Theorem \ref{thm:lct_slice_graph} $(n\sim 2n$, $k\sim n$, $u\sim\max\{\pi_{1}^{*}u,\pi_{2}^{*}v\}$, $o\sim (o,o)\in\mathbb{C}^{n}\times\mathbb{C}^{n}$,
$H\sim\Delta$ the diagonal of $\mathbb{C}^{n}\times\mathbb{C}^{n})$,
it follows that $A$ is regular at $((o,o))$ satisfying $\dim_{(o,o)}A=\dim_{(o,o)}(A\cap\Delta)+n$.
As $A_{1}\cap A_{2}=B$,
it follows that
$(A_{1}\times A_{2})\cap\Delta=A\cap\Delta$,
which implies
\begin{equation}
\label{equ:151024a}
\dim_{(o,o)}A
=\dim_{(o,o)}(A\cap\Delta)+n
=\dim_{(o,o)}((A_{1}\times A_{2})\cap\Delta)+n.
\end{equation}

Note that $A_{1}\times A_{2}\supseteq A$ and equality \ref{equ:151024a} holds,
then it follows that
$\dim_{(o,o)}(A_{1}\times A_{2})\geq\dim_{(o,o)}A=\dim_{(o,o)}((A_{1}\times A_{2})\cap\Delta)+n$.
As $\Delta$ is regular,
then it is clear that
$\dim_{(o,o)}(A_{1}\times A_{2})\leq\dim_{(o,o)}((A_{1}\times A_{2})\cap\Delta)+n$,
which implies
$\dim_{(o,o)}(A_{1}\times A_{2})=\dim_{(o,o)}((A_{1}\times A_{2})\cap\Delta)+n=\dim_{(o,o)}A$.
Note that $(A,(o,o))$ is regular and $A_{1}\times A_{2}$ is irreducible at $(o,o)$ ($A_{1}$ and $A_{2}$ are both irreducible at $o$),
then we obtain $A=A_{1}\times A_{2}$,
which implies $A_{1}$ and $A_{2}$ are both regular.

By the transversality between $A_{1}\times A_{2}=A$ and $\Delta$ at $(o,o)$ and Remark \ref{lem:linear},
it follows that $2n=\dim(T_{A_{1}\times A_{2},(o,o)}+T_{\Delta,(o,o)})=\dim T_{A_{1}\times A_{2},(o,o)}+\dim T_{\Delta,(o,o)}-\dim(T_{A_{1}\times A_{2},(o,o)}\cap T_{\Delta,(o,o)})
=(\dim T_{A_{1},o}+\dim T_{A_{2},o})+n-\dim(T_{A_{1},o}\cap T_{A_{2},o})=\dim(T_{A_{1},o}+ T_{A_{2},o})+n$.
It is clear that $\dim(T_{A_{1},o}+T_{A_{2},o})=n$,
then we prove Proposition \ref{prop:lct_add_graph}.

\vspace{.1in} {\em Acknowledgements}.
The author addresses would like to sincerely thank my advisor, Professor Xiangyu Zhou,
for bringing me to the theory of multiplier ideal sheaves and for his
valuable help to me in all ways.

The author would also like to sincerely thank
Professor Jean-Pierre Demailly for giving series of
talks on related topics at CAS and PKU
and sharing his related recent work,
Professor Liyou Zhang for pointing out some typos,
and Zhenqian Li for helpful comments.

\bibliographystyle{references}
\bibliography{xbib}

\end{document}